\documentclass[12pt]{amsart}
\usepackage{amsrefs}
\usepackage{amssymb}
\usepackage{amsmath}
\usepackage{proof}
\usepackage{centernot}
\usepackage{enumerate}
\usepackage{graphicx}
\usepackage{fullpage}
\usepackage{ragged2e}
\usepackage{blindtext}
\usepackage{cleveref}
\usepackage{amsaddr}

\newtheorem{theorem}{Theorem}

\newtheorem{corollary}[theorem]{Corollary}

\theoremstyle{definition}
\newtheorem{definition}[theorem]{Definition}

\theoremstyle{remark}
\newtheorem{remark}[theorem]{Remark}

\newtheorem{example}[theorem]{Example}

\theoremstyle{plain}
\newtheorem*{theorem*}{Theorem}
\newtheorem*{lemma*}{Lemma}
\newtheorem*{proposition*}{Proposition}

\theoremstyle{definition}
\newtheorem*{definition*}{Definition}

\theoremstyle{remark}

\def\dotminus{\mathbin{\ooalign{\hss\raise1ex\hbox{.}\hss\cr
  \mathsurround=0pt$-$}}}

\renewcommand{\:}{\colon}

\newcommand{\suchthat}{\,|\,}
\newcommand{\dash}{\operatorname{-}}

\newcommand{\union}{\cup}

\newcommand{\proves}{\vdash}

\newcommand{\PRA}{\mathsf{PRA}}
\newcommand{\PRS}{\mathsf{PRS}}

\newcommand{\ZF}{\mathsf{ZF}}

\newcommand{\ISigma}{\mathsf{I\Sigma}}

\renewcommand{\And}{\wedge}
\newcommand{\Or}{\vee}
\newcommand{\Yields}{\Rightarrow}
\newcommand{\YIELDS}{\mathop{\;\Longrightarrow\;}}

\newcommand{\Implies}{\mathop{\, \rightarrow\,}}

\newcommand{\Truth}{\top}
\newcommand{\Falsehood}{\bot}

\newcommand{\Proves}{\vdash}

\newcommand{\OR}{\bigvee}
\newcommand{\AND}{\bigwedge}

\newcommand{\Exists}{\exists}
\newcommand{\Forall}{\forall}

\renewcommand{\S}{\mathcal S}

\newcommand{\NN}{\mathbb N}

\newcommand{\RR}{\mathbb R}

\newcommand{\iI}{\mathfrak I}

\newcommand{\kK}{\mathfrak K}
\newcommand{\lL}{\mathfrak L}
\newcommand{\mM}{\mathfrak M}

\begin{document}

\title{A complete system of deduction for $\Sigma$ formulas}
\author{\footnotesize Andre Kornell}
\address{ \footnotesize Department of Mathematics\\ University of California, Davis}
\email{kornell@math.ucdavis.edu}

\begin{abstract}
The $\Sigma$ formulas of the language of arithmetic express semidecidable relations on the natural numbers. More generally, whenever a totality of objects is regarded as incomplete, the $\Sigma$ formulas express relations that are witnessed in a completed portion of that totality when they hold. In this sense, the $\Sigma$ formulas are more concrete semantically than other first-order formulas. 

We describe a system of deduction that uses only $\Sigma$ formulas. Each axiom, an implication between two $\Sigma$ formulas, is implemented as a rewriting rule for subformulas. We exhibit a complete class of logical axioms for this system, and we observe that a distributive law distinguishes classical reasoning from intuitionistic reasoning in this setting. 

Skolem's theory $\mathsf{PRA}$ of primitive recursive arithmetic can be formulated in our deductive system. In Skolem's system, free variables are universally quantified implicitly, but in our formulation, free variables act as parameters to the deduction. In this sense, our formulation is more explicitly finitistic. Furthermore, most of our results are themselves finistic, being theorems of $\mathsf{PRA}$. In particular, appending our main theorem to a celebrated chain of reductions from reverse mathematics, we find that an implication of $\Sigma$ formulas is derivable in $\mathsf{WKL}_0$ if and only if there is a deduction from the antecedent to the consequent in our formulation of $\mathsf{PRA}$.
\end{abstract}

\maketitle

Some mathematicians taking a potentialist view of the totality $\NN$ of natural numbers or the totality $\RR$ of real numbers or the totality $\mathbf V$ of pure sets reject the validity of classical reasoning for that structure. Such mathematicians are wary of quantification over an incomplete totality, in various terms. For example, Weaver analyzes quantification via the notion of surveyability, the abstract possibility of sorting through the totality \cite{Weaver09} \cite{Weaver15}, while Feferman expresses the view that only quantification over a complete totality preserves the definiteness of a property \cite{Feferman09} \cite{Feferman09A}, which he variously analyzes as semantic bivalence, or the expressibility of that property by an operation to the set of truth values $\{\Truth, \Falsehood\}$. Roughly speaking, both Weaver and Feferman advocate for the use of intuitionistic logic for formulas that quantify over incomplete totalities. This attitude goes back to at least Pozsgay \cite{Pozsgay72}, who investigated a variant of $\ZF$ in which the law of the excluded middle is postulated only for formulas whose quantifiers range over sets. We follow these authors in reasoning about an incomplete structure as a whole, in contrast to an approach that reasons about an incomplete structure in terms of its completed substructures \cite{HamkinsLinnebo17} \cite{Hamkins18}.

The attitude of the present paper is that universal quantification over an incomplete totality is impossible, meaningless. Completeness is not a formal notion, and there are conceptual models for which the word ``complete'' might sensibly be interpreted to opposite effect. For example, the incompleteness of a totality might be treated modally, with universal quantification expressible using the necessity operator. However, it is also sensible to call a totality complete if its elements are all epistemologically accessible in some ideal sense. This is more or less our interpretation of the word ``complete''. 

Intuitionistic reasoning is not justified per se, if any of its formulas are meaningless. We formalize a mode of deduction for the formulas that we take to be meaningful for our understanding of incomplete totalities. Our conceptual model is based on ordinary computability: if we take the totality $\NN$ of natural numbers to be incomplete, then our meaningful formulas are essentially the computably semidecidable predicates. In the general case, we reason about a structure whose universe is not assumed to be a complete totality. To accommodate universal quantification over its complete fragments, the structure is equipped with a binary relation we notate $\in$; we require that the totality of points related to any given point of the structure be complete. This condition is similar in kind to the requirement that equality relation $=$ be an equivalence relation respected by each formula.

We take all atomic formulas to be meaningful, and we take take the class of meaningful formulas to be closed under conjunction, disjunction, bounded universal quantification, and unbounded existential quantification. We allow unbounded existential quantification because an existential claim does not require epsitemological access to all the elements of universe, but just to one. The logical complement of an existential claim is a universal claim, which we take to be meaningless; hence, negation is not one of our allowed connectives. However, we do take the relations $=$ and $\in$ to be abstractly decidable, so the vocabulary of the structures we consider includes binary relation symbols $\neq$ and $\not \in$, which denote the complements of relations $=$ and $\in$. Thus, we take the $\Sigma$ formulas in the language of set theory to be meaningful in themselves; for finitary set theory, this attitude goes back to Takahashi \cite{Takahashi}.

For precision, we specify a first-order formula to be a formula built up using conjunction, disjunction, universal quantification, existential quantification, and implication, starting from the atomic formulas, including the Boolean sentence letters $\Truth$ and $\Falsehood$. We render the quantifiers as $\forall v\:$ and $\exists v\:$, using the semicolon as a visual separator, with the convention that quantifiers bind less tightly than each connective except implication. We identify the meaningful first-order formulas with the $\Sigma$ formulas defined here:

\begin{definition}
Let $\S$ be a vocabulary consisting of finitary function and relation symbols, including a distinguished binary relation symbol $\not \in$. We define $\Sigma(\S)$ to be the class of first-order formulas in the vocabulary $\S$ whose only connectives are conjunction, disjunction, existential quantification, and universal quantification, with each occurrence $\forall v$ of a universal quantifier quantifying a formula of the form $v  \not \in t \Or \phi$ with $v$ not free in $t$. We abbreviate this formula as $\forall v \in t\: \phi$.
\end{definition}

The usual definition of $\Sigma$ formulas allows negation where it does not occur above an unbounded existential quantifier. Our definition effectively confines negation to literals to ensure that $\Sigma$ formulas are positive in their atomic subformulas. Furthermore, we control which atomic formulas may occur negated by explicitly adding negated relation symbols, or neglecting to do so. The $\Sigma$ formulas of the language of set theory are essentially $\Sigma(\S)$ formulas for $\S$ the vocabulary consisting of binary relation symbols $=$, $\neq$, $\in$, and $\not \in$. The $\Sigma$ formulas of the language of arithmetic are essentially $\Sigma(\S)$ formulas for $\S$ the vocabulary consisting of binary relation symbols $=$, $\neq$, $<$, $\not <$, $0$, $1$, $+$, and $\cdot$, where we take $<$ and $\not <$ as alternative renditions of $\in$ and $\not \in$. In both of these settings, $\Sigma$ formulas are equivalent to $\Sigma_1$ formulas, which are characterized by the additional condition that the only occurrence of unbounded existential quantification is at the front of the formula. Generally, this equivalence requires a collection principle such as
$$\forall t \in u\: \exists v\: \phi \YIELDS \exists w\: \forall t \in u\: \exists v \in w\: \phi.$$

\begin{definition}
Let $\S$ be a vocabulary consisting of finitary function and relation symbols, including a distinguished binary relation symbol $\not \in$. We define a $\Sigma(\S)$ theory to be a class of implications $\phi \Yields \psi$, with $\phi$ and $\psi$ both $\Sigma(\S)$ formulas.
\end{definition}

If $\phi$ is a first order formula whose free variables are $v_1, \ldots, v_n$, we define its universal closure $\overline \phi$ to be the formula $\forall v_1\: \cdots \forall v_n\: \phi$. We assume that our variables are canonically linearly ordered, resolving the ambiguity in the order of the universal quantifiers. If $T$ is a $\Sigma(\S)$ theory, we define the class $\overline T$ to consist of universal closures of the implications in $T$. A model of $T$ is just a model of $\overline T$ in the ordinary sense.

Three arguments form the core of this paper. For each $\Sigma(\S)$ theory $T$, and all $\Sigma(\S)$ formulas $\phi$ and $\psi$, we use the cut-elimination theorem to reduce derivations of $\overline T\Proves \overline {\phi \Yields \psi}$ in Gentzen's $\lL\kK$, to derivations of $\phi \Proves \psi$ in $\lL\kK_\Sigma(T)$, a sequent calculus for $\Sigma$ formulas that allows cuts according to the axioms of $T$. Second, we use deep inference to reduce derivations of $\phi \Proves \psi$ in $\lL\kK_\Sigma(T)$ to deductions in the rewriting system $\mathrm{RK}_\Sigma(T)$. Third, we establish the soundness of the rewriting system $\mathrm{RK}_\Sigma(T)$ as a consequence of the positivity of our connectives.

The derivations we consider are all finite objects, so we will assume that the vocabulary $\S$ is finite. With eye towards applications in the foundations of mathematics, we claim many numbered results from the theory $\Sigma^0_1\dash\mathsf{PA}$ \cite{Simpson}*{remark I.7.6}, also called $\mathsf{I\Sigma}$.

\begin{theorem}[$\mathsf{I\Sigma}$]\label{theorem A}
Let $\S$ be a finite set of finitary function and relation symbols, including a distinguished binary relation symbol $\not \in$. Let $T$ be a finite set of formulas that are implications between $\Sigma(\S)$ formulas. Let $\phi$ and $\psi$ also be $\Sigma(\S)$ formulas. The following are equivalent:
\begin{enumerate}
\item There is a derivation of the sequent $\overline T \Proves \overline{\phi\Yields \psi}$ in the sequent calculus $\lL\kK$.
\item There is a derivation of the sequent $\phi \Proves \psi$ in the sequent calculus $\lL \kK_\Sigma(T)$.
\item There is a deduction from $\phi$ to $\psi$ in the rewriting system $\mathrm{RK}_\Sigma(T)$.
\end{enumerate}
\end{theorem}

Parsons's theorem is the proposition that any $\Pi^0_2$ theorem of $\mathsf{I \Sigma}$ is also a theorem of $\mathsf{PRA}$, so in particular theorem \ref{theorem A} is a theorem of $\mathsf{PRA}$. We refer here primarily to the first-order formulation of $\mathsf{PRA}$ given in Simpson's \textit{Subsystems of Second Order Arithmetic} \cite{Simpson}*{section IX.3}; Parson's theorem is a special case of \cite{Simpson}*{theorem IX.3.16}. Skolem's original formulation of $\PRA$ is a quantifier-free system; by a standard application of Herbrand's theorem, any $\Pi^0_2$ theorem of $\mathsf{I \Sigma}$ is also a theorem of this quantifier-free system in the sense that the theorem provably admits a primitive recursive Skolem function. For each of the implications $(1) \Yields (2)$, $(2) \Yields (3)$, and $(3) \Yields (1)$, the primitive recursive Skolem function is implicit in our proof of theorem \ref{theorem A}. Thus, we have a primitive recursive algorithm for reducing derivations in $\lL \kK$ with nonlogical axioms from $T$, to deductions in the rewriting system $\mathrm{RK}_\Sigma(T)$. Skolem's $\mathsf{PRA}$ is significant as a formalization of Hilbert's finitism, as prominently argued by Tait \cite{Tait}. See example \ref{example PRA} for further discussion of these systems. 

Putting the classical completeness theorem through the equivalence $(1) \Leftrightarrow (3)$, we obtain the following variant.

\begin{corollary}\label{corollary B}
Let $\S$ be any set of finitary function and relation symbols, including a distinguished binary relation symbol $\not \in$. Let $T$ be any set of formulas that are implications between $\Sigma(\S)$ formulas.  Let $\phi$ and $\psi$ be $\Sigma(\S)$ formulas. There is a deduction from $\phi$ to $\psi$ in the rewriting system $\mathrm{RK_\Sigma}$ if and only if every model of $\overline T$ also satisfies $\overline{\phi \Yields \psi}$.
\end{corollary}

This corollary can also be obtained directly via the standard term-model construction, using the Rauszer-Sabalski lemma \cite{RauszerSabalski} in place of the Rasiowa-Sikorski lemma.

The deductive system $\lL \kK$ is Gentzen's sequent calculus with the cut rule, except we allow the Boolean sentence letters $\Truth$ and $\Falsehood$. It is depicted in figure \ref{figure 1}. We obtain cut elimination for the this variant as a quick corollary of the cut-elimination theorem for the original system, by adding $\Truth$ and $\Falsehood$ to the antecedent and succedent respectively of each sequent in a given derivation. The cut-elimination theorem is known to be provable in $\mathsf{I\Sigma}$; see for example the proof in reference \cite{Buss}*{section I.2.4}. Our formulation omits mention of negation, which does appear in Gentzen's original system, because a negated formula $\neg \phi$ can be rendered $\phi \Yields \Falsehood$. For brevity, we have combined the contraction and exchange structural rules of Gentzen's original formulation of $\lL\kK$ into a single ``multiset rule''.

The deductive system $\lL\kK_\Sigma(T)$ is essentially a restriction of $\lL \kK$ to $\Sigma$ formulas, with the axioms $T$ showing up as cuts in the deduction:

\begin{definition}\label{definition LSigma}
Let $\S$ be any set of finitary function and relation symbols, including a distinguished binary predicate $\not \in$. Let $T$ be any set of forumas that are implications between $\Sigma(\S)$ formulas. The sequents of the deductive system $\lL\kK_\Sigma(T)$ are those of $\lL\kK$, with each formula a $\Sigma(\S)$ formula. The rules of the deductive system $\lL\kK_\Sigma(T)$ are also those of $\lL\kK$, but with a cut rule for each axiom $\phi(x_1, \ldots, x_n) \Yields \phi'(x_1, \ldots, x_n)$ in $T$ and terms $t_1, \ldots, t_n$:
$$
\infer{\Gamma, \Gamma' \Proves \Delta, \Delta'}{\Gamma \proves \Delta, \phi(t_1, \ldots, t_n) & \Gamma', \phi'(t_1, \ldots, t_n) \Proves \Delta' }
$$
The resulting system is depicted in figure \ref{figure 2}. The rule for implication is automatically excluded because it produces excluded sequents. Thus, the system $\lL\kK_\Sigma(\emptyset)$ is cut-free, and the system $\lL\kK_\Sigma(\{ \phi \Yields \phi \suchthat \phi \in \Sigma(\S)\})$ has just the ordinary cut rule.
\end{definition}

In the rewriting system $\mathrm{RK}_\Sigma(T)$, each cut rule for an axiom of $T$ is implemented as a rewriting rule. The other rules of the system $\lL\kK_\Sigma(T)$ are likewise implemented as rewriting rules.

\begin{definition}
Same assumptions as for \cref{definition LSigma}, above. A deduction in the system $\mathrm{RK}_\Sigma(T)$ is a sequence of $\Sigma(\S)$ formulas, with each succeeding formula obtained from its predecessor by an inference rule. For each implication $\phi(x_1, \ldots, x_n) \Yields \phi'(x_1, \ldots, x_n)$ that is an axiom of $T$ or a logical axiom in figure \ref{figure 3}, and for all terms $t_1, \ldots, t_n$, the system $\mathrm{RK}_\Sigma(T)$ has the rule
$$
\infer{\Xi(P/\phi'(t_1, \ldots,t_n))}{\Xi(P/\phi(t_1, \ldots,t_n))},
$$
where $P$ is a schematic variable in $\Xi$ in place of a subformula, and the slash notation expresses substitution. The variables in the terms $t_1, \ldots, t_n$ may become bound in the formula $\Xi(P/\phi(t_1, \ldots,t_n))$, or equivalently in the formula $\Xi(P/\phi'(t_1, \ldots,t_n))$.
\end{definition}

The system $\mathrm{RK}_\Sigma(T)$ is a deep inference rewriting system, so called because its inferences may be applied to subformulas \cite{Schutte77} \cite{Brunnler04}. The system $\mathrm{RK}_\Sigma(T)$ is very much in the vein of Beklemishev's rewriting system for strictly positive logics \cite{Beklemishev}, with which it shares the unusual feature that even the the nonlogical axioms of the theory are interpreted as rewriting rules. The two systems were developed independently of one another.

\begin{proof}[Proof of $(1) \Yields (2)$.]
Let $\S$ be a finite set of finitary function and relation symbols, including a distinguished binary predicate $\not \in$, and let $\overline T \Proves  \overline{\phi \Yields \psi}$ be a sequent derivable in $\lL\kK$, for $\phi$ and $\psi$ both $\Sigma(\S)$ formulas, and $T$ a finite set of implications between $\Sigma(\S)$ formulas. The sequent $\overline T, \phi \Proves \psi$ is evidently also derivable in $\lL \kK$, so by the cut-elimination theorem, it has a cut-free derivation $d$.

By induction, implication can only occur on the left of a sequent in $d$, and furthermore the left universal quantifier rule is the only rule that acts on a formula in which implication occurs. We can adjust the derivation $d$ so that all universal quantifiers above an implication symbol are placed immediately after the formation of that implication. Formally, we may define a rank that counts the number of nonclosed implications in $d$ that are in the premise of a step that isn't the introduction of a universal quantifier above that implication. By induction on this rank, we can show that there is a cut-free derivation $d'$ of $\overline T, \phi \Proves \psi$ in $\lL\kK$ in which all implications are universally closed immediately after their formation.

By simply removing all formulas containing implication from $d'$, we obtain a derivation of $\phi \Proves \psi$ in $\lL\kK_\Sigma(T)$: Each segment of $d'$ that consists in the formation of an implication followed by its universal closure is transformed into an application of the $T$-cut rule followed by steps in which the premise is equal to the conclusion. These trivial steps may be painlessly removed. Each formula in the resulting derivation is also a formula of the derivation $d'$, and is therefore a subformula of a formula in $\overline T \union \{\phi, \psi\}$, in the inclusive sense appropriate to the statement of the subformula property of cut-free derivations. By induction, each subformula of a $\Sigma(\S)$ formula is itself a $\Sigma(\S)$ formula, because in every formula of the form $\forall v \in s \: \psi$, the term $s$ is formally in the scope of the quantifier $\forall v$, so our derivation does consist of $\Sigma(\S)$ formulas.
\end{proof}

\begin{proof}[Proof of $(2) \Yields (3)$, forward direction]
Let $\S$ be a finite set of finitary function and relation symbols, including a distinguished binary predicate $\not \in$, and let $T$ be a finite set of implications between $\Sigma(\S)$ formulas. For simplicity, we first consider the argument for $\lL\kK_\Sigma(T)$ and $\mathrm{RK}_\Sigma(T)$ without the symbols $\Truth$ and $\Falsehood$. In this case, every sequent $\Gamma \Proves \Delta$ in any derivation in $\lL\kK_\Sigma(T)$ has a nonempty antecedent $\Gamma$ and a nonempty succedent $\Delta$.

We call a sequent $\Gamma \Proves \Delta$ of $\Sigma(\S)$ formulas ``deducible'' if both $\Gamma$ and $\Delta$ are nonempty and there is a deduction from $\AND \Gamma$ to $\OR \Delta$ in the rewriting system $\mathrm{RK}_\Sigma (T)$. The notation $\AND \Gamma$ abbreviates the conjunction $\gamma_1 \And \cdots \And \gamma_n$ associated from the left, where $\Gamma = \gamma_1, \ldots, \gamma_n$, and analogously the expression $\OR \Delta$ abbreviates the disjunction of the formulas in $\Delta$.

Fix a derivation in $\lL\kK_\Sigma(T)$ of some sequent $\phi_0 \Proves \psi_0$. We show that every sequent in this fixed derivation is deducible, by $\Sigma_1$ induction on the length of its derivation. It is sufficient to show for each rule of $\lL\kK_\Sigma(T)$ that if the premise sequents are deducible, then so is the conclusion sequent. Before we examine the rules of $\lL\kK_\Sigma(T)$, we observe that the rules of $\mathrm{RK}_\Sigma(T)$ allow us to commute and reassociate conjuncts (and disjuncts) as follows:
$$\phi \And \psi \YIELDS^{(5)} (\phi \And \psi) \And (\phi \And \psi) \YIELDS^{(4)} \psi \And (\phi \And \psi) \YIELDS^{(3)} \psi \And \phi$$
$$\phi \And (\psi \And \chi) \YIELDS^{(5)} (\phi \And (\psi \And \chi)) \And (\phi \And (\psi \And \chi)) \YIELDS^{(3)} (\phi \And \psi) \And (\phi \And (\psi \And \chi)) \YIELDS^{(4)}  (\phi \And \psi) \And (\phi \And \chi) \YIELDS^{(4)} (\phi \And \psi) \And \chi$$

\vspace{3mm}

Using rules (5) and (3) to duplicate and recombine instances of various conjuncts, and likewise for disjuncts, we find that if $\Gamma \Proves \Delta$ is deducible, and $\Gamma'$ consists of the same formulas as $\Gamma$, and $\Delta'$ consists of the same formulas as $\Delta$, then $\Gamma' \Proves \Delta'$ is also deducible. This addresses the multiset rule. If $\Gamma \Proves \Delta$ is deducible, we also find that $\Gamma, \phi \Proves \Delta$ is deducible using rule $(3)$, and likewise $\Gamma \Proves \Delta, \phi$ is deducible using rule (6). And certainly, $\phi \Proves \phi$ is trivially deducible for any formula $\phi$. This leaves the $T$-cut rule, and the logical rules.

Let $\phi(x_1, \ldots, x_n) \YIELDS \psi(x_1, \ldots, x_n)$ be an axiom of $T$, and assume that the sequents $\Gamma \Proves \Delta, \phi(t_1, \ldots, t_n)$ and $\Gamma, \psi(t_1, \ldots, t_n) \Proves \Delta$ are deducible, for some terms $t_1, \ldots t_n$. We obtain a proof of the sequent $\Gamma \Proves \Delta$ in the following way:
\begin{align*}
\AND \Gamma & \YIELDS^{(5)} \AND \Gamma \And \AND \Gamma \YIELDS \AND \Gamma \And \left( \OR \Delta  \Or \phi(t_1, \ldots t_n)\right) 
\\ & \YIELDS^{(9)}\left(\AND \Gamma \And \OR \Delta\right) \Or \left(\AND \Gamma \And \phi(t_1, \ldots, t_n)\right) 
\\ &\YIELDS^{(0)} \left(\AND \Gamma \And \OR \Delta\right)  \Or \left(\AND \Gamma \And \psi(t_1, \ldots, t_n)\right)
\\ & \YIELDS^{(4)}\left(\AND \Gamma \And \OR \Delta\right) \Or \OR \Delta   \YIELDS^{(8)} \OR \Delta \Or \OR \Delta \YIELDS \OR \Delta
\end{align*}
This addresses the $T$-cut rule.

Only the logical rules remain. If $\Gamma, \phi \Proves \Delta$ is a deducible sequent then so are the sequents $\Gamma, \phi \And \psi \Proves \Delta$ and $\Gamma, \psi \And \phi \Proves \Delta$, using rules (3) and (4). The right disjunction rules are addressed likewise. If the sequents $\Gamma \Proves \Delta, \phi$ and $\Gamma \Proves \Delta, \psi$ are both deducible, then so is the sequent $\Gamma \proves \Delta, \phi \And \psi$:
$$ \AND \Gamma \YIELDS \AND \Gamma \And \AND \Gamma
\YIELDS \left( \OR \Delta \Or \phi\right) \And \left(  \OR \Delta \Or \psi\right) \YIELDS \OR \Delta \Or (\phi \And \psi)
$$
The final implication is a distributive law that is known to be equivalent to the distributive law (9) over rules (3)-(8), which are the axioms of a lattice. The left disjunction rule is addressed likewise.

If the sequent $\Gamma \Proves \Delta, \phi(v/t)$ is deducible, then so is $\Gamma \Proves \Delta, \exists v\: \phi$, using rule (10). If the sequent $\Gamma, t \not\in s \Or \phi(v/t) \Proves \Delta$ is deducible, then so is $\Gamma, \forall v \in s\: \phi \proves \Delta$, using rule (13). If the sequent $\Gamma, \phi(v/w) \Proves \Delta$ is deducible for some variable $w$ not free in $\Gamma$, $\Delta$, or $\phi$, then so is  $\Gamma, \exists v\: \phi \Proves \Delta$:
\begin{align*}
\AND \Gamma \And \exists v\: \phi & \YIELDS
\AND \Gamma \And \exists v\: \phi(v/w)(w/v) \\ & \YIELDS ^{(10)}
\AND \Gamma \And \exists v\: \exists w\: \phi(v/w) \\ & \YIELDS^{(11)}
\AND \Gamma \And \exists w\: \phi(v/w) \\ & \YIELDS^{(12)}
\exists w\: \AND \Gamma \And \phi(v/w) \\ & \YIELDS^{(11)}
\exists w\: \OR \Delta \YIELDS \OR \Delta
\end{align*}
Finally, if $\Gamma \Proves \Delta, w \not \in s \Or \phi(v/w)$ is deducible for some variable $w$ not in free in $\Gamma$, $\Delta$, $\phi$, or $s$, then so is $\Gamma \Proves \Delta, \forall v \in s\: \phi$:
\begin{align*}
\AND \Gamma & \YIELDS^{(14)}
\forall w \in s\: \AND \Gamma \\ & \YIELDS
\forall w \in s\: \OR \Delta \Or (w \not \in s \Or \phi(v/w))\\ & \YIELDS^{(15)}
\OR \Delta \Or \forall w \in s\: w \not \in s \Or \phi(v/w) \\ & \YIELDS^{(14a)}
\OR \Delta \Or \forall w \in s\: \phi(v/w) \\ & \YIELDS^{(14)}
\OR \Delta \Or\forall v \in s\: \forall w \in s\: \phi(v/w)  \\ & \YIELDS^{(13)}
\OR \Delta \Or\forall v\in s\: v \not \in s \Or \phi \\ &\YIELDS^{(14a)}
\OR \Delta \Or\forall v \in s\: \phi
\end{align*}
The expression $\OR \Delta$ is undefined when $\Delta$ is empty, but in this case the obvious simple variant of the deduction suffices. Note that this variant proof does not appeal to rule (15) of $\mathrm{RK}_\Sigma(T)$.

We now consider the argument for $\lL\kK_\Sigma(T)$ and $\mathrm{RK}_\Sigma$ with $\Truth$ and $\Falsehood$. A derivation in $\lL\kK_\Sigma(T)$ that uses the rules for $\Truth$ and $\Falsehood$ may be easily modified to a derivation that does not use these rules, by adding $\Truth$ to the antecedent of every sequent and adding $\Falsehood$ to the consequent of every sequent. Applying the previous argument, we obtain a deduction from $\Truth \And \phi_0$ to $\Falsehood \Or \psi_0$ in $\mathrm{RK}_\Sigma(T)$; axioms (1) and (2) then yield a deduction from $\phi_0$ to $\psi_0$.
\end{proof}

\begin{proof}[Proof of $(3) \Yields (1)$.]
Let $\S$ be a finite set of finitary function and relation symbols, including a distinguished binary predicate $\not \in$, and let $T$ be a finite set of implications between $\Sigma(\S)$ formulas. For each axiom $\phi(x_1, \ldots, x_n) \Yields \psi(x_1, \ldots, x_n)$ in $T$, and all terms $t_1, \ldots, t_n$, the sequent $\overline T, \phi(t_1, \ldots, t_n) \Proves \psi(t_1, \ldots, t_n)$ is derivable in $\lL \kK$. In fact, it is derivable in the sequent calculus $\lL \iI$ for intuitionistic logic, obtained from $\lL \kK$ by requiring that every sequent $\Gamma \proves \Delta$ have at most one formula in the succedent $\Delta$: $$ \infer{\overline T, \phi(t_1, \ldots, t_n) \Proves \psi(t_1, \ldots, t_n)}{\infer{\overline{\phi \Yields \psi}, \phi(t_1, \ldots, t_n) \Proves \psi(t_1, \ldots, t_n)}{ \infer{ \phi(t_1, \ldots, t_n) \Yields \psi(t_1, \ldots, t_n), \phi(t_1, \ldots, t_n) \Proves \psi(t_1, \ldots, t_n)}{ \phi(t_1, \ldots t_n) \Proves \phi(t_1, \ldots, t_n) & \psi(t_1, \ldots t_n) \Proves \psi(t_1, \ldots, t_n)}}}
$$

We may similarly derive the sequent $\overline T, \chi \Proves \chi'$ in $\lL \kK$ when $\chi \Yields \chi'$ is one of the logical rules (1)-(15) of $\mathrm{RK}_\Sigma(T)$. In place of appealing to an axiom of $T$ during the derivation, we apply the corresponding logical rule of $\lL \kK$. For example, for an instance $\phi \And \psi \Yields \phi$ of rule (3), we may derive the sequent $\overline T, \phi \And \psi \Proves \phi$ as follows:
$$
\infer{ \overline T, \phi \And \psi \Proves \phi}{\infer{\phi \And \psi \Proves \phi}{\phi \Proves \phi}}
$$
We may be certain that such derivations exist for all the logical rules of $\mathrm{RK}_\Sigma(T)$, by the completeness of the system $\lL \kK$; in fact, all these derivations are fairly short. The only such derivation not in $\lL \iI$ is rule (15). An implication $\forall v \in s \: (\chi \Or \phi) \YIELDS \chi \Or (\forall v \in s \: \phi)$ is generally not intuitionistically valid. The other derivations either follow the two-step pattern of the above example for $\phi \And \psi \Yields \phi$, or are depicted in figures \ref{figure 4} -- 7.

If a sequent $\overline T, \chi \Proves \chi'$ is derivable in $\lL \kK$, then so is the sequent $\overline T, \Xi(P/\chi) \Proves \Xi(P/\chi')$, for any formula $\Xi$ with a single schematic variable $P$ in place of a subformula. We derive the sequent $\overline T, \Xi(P/\chi) \Proves \Xi(P/\chi')$ recursively, by following the logical structure of the formula $\Xi$. We are given a derivation of $\overline T, \Xi(P/\chi) \Proves \Xi(P/\chi')$ when $\Xi$ is simply $P$, and we derive $\overline T, \Xi(P/\chi) \Proves \Xi(P/\chi')$ by the weakening rule when $\Xi$ is an atomic formula other than $P$. When $\Xi$ is not atomic, we apply the appropriate logical rules. For example:
$$
\infer{\overline T, \Xi_0(P/\chi)\Or  \Xi_1(P/\chi)  \Proves\Xi_0(P/\chi') \Or \Xi_1(P/\chi')}{
\infer{\overline T, \Xi_0(P/\chi) \Proves \Xi_0(P/\chi') \Or \Xi_1(P/\chi')}{
{\overline T, \Xi_0(P/\chi) \proves \Xi_0(P/\chi')}
} &
\infer{\overline T, \Xi_1(P/\chi) \Proves \Xi_0(P/\chi) \Or \Xi_1(P/\chi') }{
{\overline T, \Xi_1(P/\chi) \proves \Xi_1(P/\chi')}
}
}
$$
All such derivations occur entirely within $\lL \iI$.

A deduction in the rewriting system $\mathrm{RK}_\Sigma (T)$ is a sequence of formulas $\xi_0, \ldots, \xi_n$ such that for each step $i \in \{1, \ldots, n\}$ we can find a formula $\Xi_i$ with schematic variable P, and a rule $\chi_{i-1} \Yields \chi_i$ of $\mathrm{RK}_\Sigma(T)$, such that $\xi_{i-1}$ is $\Xi_i(\chi_{i-1})$ and $\xi_i$ is $\Xi_i(\chi_i)$. Thus, we have shown that for each step $i \in \{1, \ldots, n\}$, the sequent $\overline T, \xi_{i-1} \proves \xi_i$ is derivable in $\lL \kK$. Applying the cut rule $n-1$ times, we obtain a derivation of $\overline T, \xi_0 \Proves \xi_n$ in $\lL \kK$, and therefore of $\overline T\Proves \overline {\xi_0 \Yields \xi_n}$, as desired.
\end{proof}

We presently observe that all three parts of the argument remain sound if we restrict the systems $\lL \kK$ and $\lL \kK_\Sigma(T)$ to sequents $\Gamma \Proves \Delta$ having at most one formula in the succedent $\Delta$, provided that we also remove rule (15) from the system $\mathrm{RK}_\Sigma(T)$. Our interest in this observation is motivated by the fact that the sequent calculus $\lL \iI$, obtained by restricting $\lL \kK$ in this way, is known to formalize intuitionistic reasoning. 

\begin{definition}
A derivation in $\lL \iI$ is a derivation in $\lL \kK$ each of whose sequents $\Gamma \Proves \Delta$ has at most one formula in its succedent $\Delta$.  Similarly, a derivation in $\lL \iI_\Sigma(T)$ is a derivation in $\lL \kK_\Sigma(T)$ each of whose sequents has at most one formula in its succedent.
\end{definition}

\begin{theorem}[$\ISigma$]
Same assumptions as of \cref{theorem A}. The following are equivalent:
\begin{enumerate}
\item There is a derivation of the sequent $\overline T \Proves \overline{\phi \Yields \psi}$ in the sequent calculus $\lL \iI$.
\item There is a derivation of the sequent $\phi \Proves \psi$ in the sequent calculus $\lL \iI_\Sigma(T)$.
\item There is a deduction from $\phi$ to $\psi$ in the rewriting system $\mathrm{RK}_\Sigma(T)$ that does not appeal to rule (15).
\end{enumerate}
\end{theorem}

\begin{proof}
Our proof of the implication $(1) \Yields (2)$ for theorem \ref{theorem A} begins by fixing a cut-free derivation of $\overline T, \phi \Proves \psi$, and then proceeds to control and suppress the formation of the sentences in $\overline T$, without altering any succedents. Thus, to conclude that this proof adapts to the sequent calculus $\lL\iI$, it is sufficient to recall that $\lL \iI$ allows the elimination of cuts.

Our proof of the implication $(2) \Yields (3)$ uses rule (15) of $\mathrm{RK}_\Sigma(T)$ in only one place: to show that if $\Gamma \Proves \Delta, w \not \in s \Or \phi(v/w)$ is deducible, then so is $\Gamma \Proves \Delta, \forall v \in s\: \phi$. If a sequent of the form $\Gamma \Proves \Delta, w \not \in s \Or \phi(v/w)$ appears in a derivation in $\mathrm{RK}_\Sigma(T)$, then the sequence $\Delta$ must be empty. In this case, a simplified proof of $\Gamma \Proves \Delta, \forall v \in s\: \phi$ is available, which does not use rule (15).

Our proof of the implication $(3) \Yields (1)$ consists essentially in showing that for each rule $\chi \Yields \chi'$ of $\mathrm{RK}_\Sigma(T)$, the sequent $\overline T, \chi \proves \chi'$ is derivable in $\lL\kK$. The only rule of $\mathrm{RK}_\Sigma(T)$ for which the corresponding sequent is not derivable using $\lL \iI$ is rule (15). The derivations of those rules of $\mathrm{RK}_\Sigma(T)$ whose derivations have height greater than $1$ are given in figures \ref{figure 4} -- 7.
\end{proof}

There is an evident correspondence between many of the rules of $\lL \kK$, $\lL\kK_\Sigma(T)$, and $\mathrm{RK}_\Sigma(T)$, which runs through all three parts of the argument. The argument remains valid if we omit the logical rules for a logical symbol such as $\Falsehood$ or $\Exists$ from all three formal systems. However, we cannot omit the rules for conjunction and disjunction in this way, because rules (3)-(9) of $\mathrm{RK}_\Sigma(T)$ play a structural, as well as a logical, role.

\begin{theorem}[$\ISigma$]\label{no symbols}
Same assumptions as of theorem 3. The statements of theorem 3 and theorem 9 remain true if we modify the formal systems $\lL \kK$, $\lL\kK_\Sigma(T)$, and $\mathrm{RK}_\Sigma(T)$ by excluding all formulas containing any symbols from a fixed subset of the symbols $\Truth$, $\Falsehood$, $\Exists$, $\Forall$ and $\not \in$.
\end{theorem}

\begin{proof}
Examine all three parts of the argument for appeals to the relevant logical rules.
\end{proof}

Recall that Johansson's minimal logic differs from intuitionistic reasoning in its rejection of \textit{ex flaso quodlibet}, the principle that everything follows from a contradiction. It may be formalized by excluding the logical rule for $\Falsehood$ from $\lL \iI$; we call the resulting system $\lL \mM$. In $\lL \mM$, the symbol $\Falsehood$ is treated just as a nonlogical nullary relation symbol. Thus, removing $\Falsehood$ as logical symbol, according to \cref{no symbols}, and adding it as a nullary relation symbol to our vocabulary $\S$, we find that derivability in minimal logic corresponds to provability in $\lL \iI(T)$ without rule (2): $\Falsehood \Yields \phi$.

\begin{corollary}[$\ISigma$]
Same assumptions as of theorem 3. Additionally, let $\mathrm{RI}_\Sigma(T)$ be $\mathrm{R K}_\Sigma(T)$ with rule (15) omitted, and let $\mathrm{RM}_\Sigma(T)$ be $\mathrm{RI}_\Sigma(T)$ with rule $(2)$ omitted.
\begin{enumerate}[(1)]
\item There is a derivation of $\overline T \Proves\overline{ \phi \Yields \psi}$ in $\lL \kK$ iff there is a deduction from $\phi$ to $\psi$ in $\mathrm{RK}_\Sigma(T)$.
\item There is a derivation of $\overline T \Proves \overline {\phi \Yields \psi}$ in $\lL \iI$ iff there is a deduction from $\phi$ to $\psi$ in $\mathrm{RI}_\Sigma(T)$.
\item There is a derivation of $\overline T \Proves \overline{\phi \Yields \psi}$ in $\lL \mM$ iff there is a deduction from $\phi$ to $\psi$ in $\mathrm{RM}_\Sigma(T)$.
\end{enumerate}
\end{corollary}

\begin{remark}
In fact, the statements of theorem 3, theorem 9, theorem 10, and corollary 11 remain true if we exclude the binary relation symbol $\not \in$, and simply allow unbounded universal quantification. Here, the bounded universal quantifier rules of the system $\lL \kK_\Sigma(T)$ are replaced with their standard variants,
$$\infer{\Gamma, \forall v\: \phi}{\Gamma, \phi(v/t) \Proves \Delta}  \qquad \infer{\Gamma \Proves \Delta, \forall v \: \phi}{\Gamma \Proves \Delta, \phi(v/w)},$$
and rules (13), (14), (14a), and (15) of $\mathrm{RK}_\Sigma(T)$ are replaced with simply:
\begin{enumerate}
\item[(13)] $\forall v\:  \phi \YIELDS \phi(v/t)$
\item[(14)] $\phi \YIELDS \forall v\: \phi$
\item[(15)] $\forall v\: \chi \Or \phi \Yields \chi \Or \forall v\: \phi$
\end{enumerate}
\end{remark}

\begin{example}\label{example PRA}
The literature contains many formulations of Skolem's theory of primitive recursive arithmetic that are equivalent in the sense that they deduce the same atomic formulas. Those that allow $\Sigma^0_0$ formulas deduce the same $\Sigma^0_0$ formulas, and those that allow $\Sigma^0_1$ formulas deduce the same $\Sigma^0_1$ formulas. Each $\Sigma^0_0$ formula can be viewed as a primitive recursive process for computing truth values; and the formulations that allow $\Sigma_0^0$ formulas deduce the equivalence of each $\Sigma_0^0$ formula with the atomic formula formalizing this primitive recursive process. Furthermore, the formulations of primitive recursive arithmetic that allow $\Sigma^0_1$ formulas are constructive in the sense that if they deduce a $\Sigma^0_1$ formula $\exists v\: \phi$, then they also deduce the $\Sigma^0_0$ formula $\phi(v/t)$ for some term $t$. Thus, one might say that the various formulations of primitive recursive arithmetic are equivalent in their $\Pi^0_2$ claims. This is exactly the logical complexity of the inferences treated in the present paper.

Skolem's original formulation of primitive recursive arithmetic uses just the $\Sigma_0^0$ formulas in the vocabulary of primitive recursive arithmetic, which consists of the equality relation symbol, the zero constant symbol, and one function symbol for each primitive recursive process \cite{Skolem67}. Goodstein found an attractive ``logic-free'' formulation, i. e., a formulation using just the atomic formulas of this vocabulary \cite{Goodstein54}. In both of these formulations, the free variables of a deduced formula are implicitly universally quantified, as evident from the induction rule, which deduces $\phi(y)$ from $\phi(0)$ and $\phi(x) \Yields \phi(S(x))$. The implication should be valid for all values of $x$, or at least those less than $y$, for the conclusion to be justified. This implicit universal quantification is absent in the deductions of the systems $\mathrm{RK}_\Sigma(T)$, in which any free variables may be interpreted as parameters. Thus, these formalizations of primitive recursive arithmetic are not finitistic in the strict sense of the present paper. However, this is a technicality. Each such deduction is finitistic if it is read contrapositively: if the conclusion of a step admits a counter example, then so must on of its premises.

We introduce a formulation of primitive recursive arithmetic as a class $\mathrm{PRA}$ of implications between $\Sigma$ formulas. The advantages of $\mathrm{RK}_\Sigma(\mathrm{PRA})$ over Skolem's original formulation are, first, that $\mathrm{RK}_\Sigma(\mathrm{PRA})$ is explicitly finitistic in the sense just discussed, and second, that the language of $\mathrm{RK}_\Sigma(\mathrm{PRA})$ includes the quantifiers, enabling a natural expression of finitistic propositions. Glossing this first advantage, any free variables in a deduction of $\mathrm{RK}_\Sigma(T)$ may be read as parameters, rather than as universally quantified variables. Glossing this second advantage, the bounded quantifiers that naturally occur in finitistic propositions are expressible by logical symbols of the language of $\mathrm{RK}_\Sigma(T)$, whereas these bounded quantifiers must be expressed using function symbols in Skolem's original formulation. Thus, the formalization of finitistic reasoning to $\mathrm{RK}_\Sigma(T)$ is more or less direct, whereas the formalization of finitistic reasoning to Skolem's original formulation involves some ``coding''. The variant of primitive recursive arithmetic introduced in this paper differs from the original in the opposite direction to Goodstein's; his variant is ``logic-free'' and our variant is ``logic-rich''.

The axioms of our formulation of primitive recursive arithmetic are essentially those of the first-order formulation currently used in the reverse mathematics community. Of course, the first-order formulation is not in itself finitistic because some of its formulas quantify over the entire class of natural numbers, which we take to be an incomplete totality in the context of finitistic mathematics. However, its axioms are essentially $\Pi^0_2$, and therefore expressible as implications between $\Sigma$ formulas. Theorem 3 implies that any such implication $\phi \Yields \psi$ deducible in the first-order system is also deducible in the rewriting system $\mathrm{RK}_\Sigma(\mathrm{PRA})$. We adopt the axioms of the first-order system given in Simpson's standard reference text \cite{Simpson}*{section IX.3}; the $\Sigma^0_0$ induction scheme must be rearranged slightly to fit the pattern. We render the distinguished binary relation symbol $\not \in$ as $\not <$, revealing its intended interpretation. We also include its logical complement $<$, on the intuition that any two natural numbers are decidably comparable. The resulting system is depicted in figure \ref{figure 9}. Combining theorem 3 with Parsons's theorem and related results in reverse mathematics, we obtain the following corollary:
\end{example}

\begin{corollary}[$\ISigma$]\label{corollary PRA}
Let $\phi$ and $\psi$ be $\Sigma$ formulas in the language of primitive recursive arithmetic, with binary relation symbols $<$, $\not <$, $=$, and $\neq$. The sentence $\overline {\phi \Yields \psi}$ is deducible in $\mathsf{WKL_0}$ if and only if there is a deduction form $\phi$ to $\psi$ in the rewriting system $\mathrm{RK}_\Sigma(\mathrm{PRA})$ whose class of nonlogical axioms $\mathrm{PRA}$ is given in figures \ref{figure 8} and \ref{figure 9}.
\end{corollary}

The system $\mathsf{WKL_0}$ \cite{Simpson}*{definition I.10.1} includes the binary relation symbol $<$. In the statement of this corollary, we identify each atomic formula $s \neq t$ with the formula $\neg s = t$, and each atomic formula $s \not < t$ with the formula $\neg s < t$. 

\begin{proof}
In light of the celebrated reduction of $\mathsf{WKL}_0$ to $\mathsf{PRA}$ \cite{Simpson}*{theorem IX.3.16}, which is known to be a theorem of $\ISigma$ \cite{Burgess}*{section 4}, it is enough to argue that the sentence $\overline{\phi \Yields \psi}$ is deducible in $\mathsf{PRA}$ if and only if it is deducible in $\mathrm{RK}_\Sigma(\mathrm{PRA})$. By theorem \ref{theorem A}, it is enough to show that the axioms of $\mathsf{PRA}$ given in Simpson's text \cite{Simpson}*{section IX.3} are provably equivalent to the axioms of $\mathrm{PRA}$ given in figures \ref{figure 8} and \ref{figure 9}, provided that we translate each atomic formula $s < t$ in the former theory as the atomic formula $ t \dotminus s \neq 0$ in the latter theory \cite{Simpson}*{definition IX.3.4}.

The axioms of $\mathrm{PRA}$ given in figure \ref{figure 8} and \ref{figure 9} are certainly all deducible in the first-order system $\mathsf{PRA}$. Conversely, referring to \cite{Simpson}*{definition IX.3.2}, the first axiom follows from (14), the second axiom follows from (12), and the third axiom follows from (9), (10), (13), and the axioms for equality (1)-(6). The fourth, fifth, sixth, and seventh axioms in Simpson's text are essentially (15), (16), (17), and (18), respectively. Finally, the eighth axiom, the primitive recursive induction scheme, follows more or less directly from (19).

At first glance, the argument appears complete, but there are easily overlooked loose ends. First, the formulation of $\mathsf{PRA}$ in Simpson's text treats equality as a logical symbol, while the rewriting system $\mathrm{RK}_\Sigma(\mathrm{PRA})$ does not; the axioms for equality are (1)-(4). Second, strictly speaking, the two systems have different vocabularies, so we are formally extending the system in Simpson's text by definitions for the binary predicate symbols $\neq$, $<$, and $\not <$. The definition of $\neq$ follows from (5) and (6), and the definition of $\not <$ will similarly follow from (7) and (8).

We have shown that every first-order order formula deducible in the system $\mathsf{PRA}$ in Simpson's text follows provably from the axioms in figure \ref{figure 9}. It remains to show that the first-order formula $x < y \Leftrightarrow y \dotminus x \neq 0$ follows from these axioms. We appeal to induction on $y$. For $y=0$, one direction of the equivalence follows by (9), and the other direction follows from the equation $0 \dotminus x =0 $, which is deducible in $\mathsf{PRA}$. At the induction step, if $x < y \Leftrightarrow y \dotminus x \neq 0$, then we deduce

$$ x< Sy \; \mathop{\Longleftrightarrow}^{(11)} \; x<y \Or x = y \; \Longleftrightarrow \; y \dotminus x \neq 0 \Or x = y \; \Longleftrightarrow \; S(y) \dotminus x \neq 0.$$

\end{proof}

\begin{example}
Our motivation includes the formalization of all foundational positions that make a distinction between complete and incomplete totalities, equally the finitistic and the transfinitistic. Having reformulated $\mathsf{\PRA}$ in $\mathrm{RK}_\Sigma$, we turn to its transfinite counterpart, Rathjen's theory $\mathrm{PRS}$ of primitive recursive set functions \cite{Rathjen}. For $\mathrm{PRS}$, the universe is the totality of all pure sets, and the vocabulary has a function symbol for each primitive recursive operation on this totality \cite{JensenKarp}. Recursion occurs over the well-founded membership relation $\in$.

If the totality of all pure sets is incomplete, then we should reason about pure sets without presuming to make claims about objects that do not all exist even in an ideal platonist sense, i. e., without quantifying universally over the whole totality. Thus, we are led to formalize our reasoning in $\mathrm{RK}_\Sigma$.

The axioms of Rathjen's theory \cite{Rathjen}*{section 6} are all $\Pi_2$, so the theory can be formalized in $\mathrm{RK}_\Sigma$. However, we provide a superficial reformulation of Rathjen's theory, which formalizes set existence principles in terms of function symbols (figure \ref{figure 10}). The intention is to present the theory in full but readable detail, and to emphasize the constructive intuition of the set existence principles. The nonoverscored function symbols in Rathjen's formulation are shorthand notations, not function symbols in the theory's vocabulary.\cite{RathjenEmail}.

While our reformulation of the \emph{axioms} of $\mathsf{PRS}$ is largely superficial, our replacement of its \emph{logic} is conceptually significant. In its original formulation, $\mathsf{PRS}$ is a first-order system analogous to the system $\mathsf{PRA}$ in Simpson's text \cite{Simpson}*{IX.3}. There is no established quantifier-free variant of $\mathsf{PRS}$, so in this instance, it is not merely a technicality that the deductions of $\PRS$ include formulas that quantify over the incomplete totality of all pure sets.

\end{example}

We now adapt Rathjen's conservation theorem for $\mathsf{PRS}$ \cite{Rathjen}*{theorem 1.4} to our system, just as we did with the extension of Parsons's theorem in corollary \ref{corollary PRA}.

\begin{corollary}
Let $\phi$ and $\psi$ be $\Sigma$ formulas in the language of set theory, with binary relation symbols $\neq$ and $\not \in$. The sentence $\overline {\phi \Yields \psi}$ is derivable in $\mathsf{KP^-}$ $+$ $\Sigma_1$-foundation $+$ $\Pi_1$-foundation if and only if there is a deduction from $\phi$ to $\psi$ in the rewriting system $\mathrm{RK_\Sigma}(\mathrm{PRS})$, whose class of nonlogical axioms $\mathrm{PRS}$ is given in figures \ref{figure 8} and \ref{figure 10}.
\end{corollary}

The system $\mathsf{KP^-}$ is defined in \cite{Rathjen}*{section 1}, as are the foundation axioms. The foundation axioms implicitly allow parameters, both in the present paper and in the reference. In the statement of this corollary, we identify each atomic formula $s \neq t$ with the formula $\neg s = t$, and each atomic formula $s \not \in t$ with the atomic formula $\neg s \in t$.

\begin{proof}
We follow the proof of corollary \ref{corollary PRA}. After appealing to Rathjen's conservation theorem for $\mathsf{PRS}$ \cite{Rathjen}*{theorem 1.4}, and to corollary \ref{corollary B}, it remains only to show that Rathjen's axioms for $\mathsf{PRS}$ \cite{Rathjen}*{section 6}, together with the usual axioms for equality, imply the same sentences of the language of set theory as our theory $\mathrm{PRS}$, given in figures \ref{figure 8} and \ref{figure 10}. To accomplish this, we argue that the two theories have logically equivalent extensions by definitions.

In full detail, the argument is tedious but straightforward. Formally, we extend each theory by definitions of the form $\overline F(x_1, \ldots, x_k) = \underline f(x_1, \ldots, x_k)$, where $\overline F$ is a function symbol of Rathjen's theory, and $\underline f$ is a function symbol of our theory. However, to enhance readability, we will render these definitions indirectly.

Extending our theory to incorporate Rathjen's theory is the easier of the two extensions. Our list below matches up with the list in \cite{Rathjen}*{section 6}:
\begin{enumerate}
\item $\overline P_{n,i} = \underline P^n_i$
\item $\overline Z = S(\Falsehood)$
\item $\overline M = C (\underline A, C(\underline A, P^1_1, P^1_1), P^1_1)$
\item $\overline{C} = C (S((x_3 \in x_4 \And x_0 = x_1) \And (x_3 \not \in x_4 \And x_0 = x_2)), C( \underline A, \underline P^4_1, P^4_2), \underline P^4_1, \underline P^4_2, \underline P^4_3, \underline P^4_4)$
\item $\overline F = C(\underline k, \underline g_1, \ldots, \underline g_n)$ \hfill if $\overline F$ is obtained by composing $\overline K$ with $\overline G_1, \ldots \overline G_n$
\item $\overline F = R( C (\underline h, C (\underline U, \underline P^{n+2}_1), \underline P_2^{n+2}, \ldots, \underline P^{n+2}_{n+2})$ \hfill if $\overline F$ is obtained from $\overline H$ by recursion
\end{enumerate}
It is clear that our theory establishes Rathjen's axioms for these defined function symbols.

The definitions for the other extension are similar but more complex. Perusing the closure properties of primitive recursive set functions \cite{JensenKarp}*{1.3}, we see that the functions of our theory are indeed all primitive recursive, and the proofs of these closure properties in \cite{JensenKarp}*{appendix I} provide definitions for the function symbols of our theory in terms of those of Rathjen's theory. Furthermore, the proofs of these closure properties can be readily formalized in Rathjen's theory.
\end{proof}

Finally, taking the $\Delta_0$-definable function $G$ in the Rathjen's \cite{Rathjen}*{theorem 1.4} to be the constant function $x \mapsto \omega$, as he suggests doing, we similarly obtain the following variant of the corollary:

\begin{corollary}
Let $\phi$ and $\psi$ be $\Sigma$ formulas in the language of set theory, with binary relation symbols $\neq$ and $\not \in$. The sentence $\overline{\phi \Yields \psi}$ is deducible in $\mathsf{KP^-}$ $+$ $\Sigma_1$-foundation $+$ $\Pi_1$-foundation $+$ Infinity if and only if there is a deduction from $\phi$ to $\psi$ in the rewriting system $\mathrm{RK}_\Sigma(T)$, whose class $T$ of nonlogical axioms is given in figures \ref{figure 8}, \ref{figure 10}, and \ref{figure 11}.
\end{corollary}

I thank John Burgess, Lev Beklemishev, Joost Joosten, Michael Rathjen, and Stephen Simpson for their patient replies to my questions.

\newpage

\begin{figure}

\begin{center} Logical Rules \end{center}

$$
\infer{\Gamma, \Falsehood \Proves \Delta}{} \qquad \infer{ \Gamma \Proves \Delta, \Truth}{}
$$

$$
\infer{\Gamma, \phi \And \psi \Proves \Delta}{\Gamma, \phi \Proves \Delta}
\qquad
\infer{\Gamma,\psi \And \phi \Proves \Delta}{\Gamma, \phi \Proves \Delta}
\qquad
\infer{ \Gamma \Proves \Delta, \phi \And \psi}{\Gamma \Proves \Delta,\phi & \Gamma \Proves \Delta, \psi}
$$

$$
\infer{\Gamma, \phi \Or \psi \Proves \Delta}{\Gamma, \phi \Proves \Delta &\Gamma, \psi  \Proves \Delta}
\qquad
\infer{\Gamma \Proves \Delta,\phi \Or \psi}{\Gamma \Proves \Delta,\phi}
\qquad
\infer{\Gamma \Proves \Delta,  \psi \Or\phi}{\Gamma \Proves \Delta,\phi}
$$

$$
\infer{\Gamma, \phi \Yields \psi \Proves \Delta, \Delta'}{\Gamma \Proves \Delta, \phi & \Gamma, \psi \Proves \Delta'}
\qquad
\infer{\Gamma \Proves \Delta, \phi \Yields \psi}{\Gamma, \phi \Proves \Delta, \psi}
$$

$$
\infer{\Gamma, \exists v\: \phi \Proves \Delta}{\Gamma, \phi(v/w) \Proves \Delta}
\qquad
\infer{\Gamma \Proves \Delta, \exists v \: \phi}{\Gamma \Proves \Delta,  \phi(v/t)}
$$

$$
\infer{\Gamma, \forall v\: \phi  \Proves \Delta}{\Gamma,  \phi(v/t) \Proves \Delta}
\qquad
\infer{\Gamma \Proves \Delta, \forall v \: \phi}{\Gamma \Proves \Delta, \phi(v/w)}
$$

\vspace{7mm}

\begin{tabular}{ccc}
Axiom Rule  & \hspace{25mm} & Weakening Rules \\
&  & \\
\infer{\phi \Proves \phi}{} & & \infer{\Gamma, \phi \Proves \Delta}{ \Gamma \Proves \Delta}
\qquad
\infer{\Gamma \Proves \Delta, \phi}{\Gamma \Proves \Delta}
\end{tabular}

\vspace{7mm}

\begin{tabular}{ccc}
Multiset Rule  & \hspace{25mm} & Cut Rule \\
 &  &  \\
$\infer{\Gamma' \Proves \Delta'}{\Gamma \Proves \Delta}$ &   & $\infer{\Gamma \Proves \Delta, \Delta'}{\Gamma \Proves \Delta, \phi & \Gamma, \phi \Proves \Delta'}$ 
\end{tabular}

\caption{the system $\lL\kK$}
\label{figure 1}

\vspace{10mm}

{\justifying \noindent For the left existential quantifier rule and the right universal quantifier rule, the variable $w$ must not be free in $\Gamma$, $\Delta$, or $\phi$. For the multiset rule, the sequent $\Gamma'$ must consists of the same formulas as the sequent $\Gamma$, and the sequent $\Delta'$ must consist of the same formulas as $\Delta$.\par}

\end{figure}

\newpage

\begin{figure}

\begin{center} Logical Rules \end{center}

$$
\infer{\Gamma, \Falsehood \Proves \Delta}{} \qquad \infer{ \Gamma \Proves \Delta, \Truth}{}
$$

$$
\infer{\Gamma, \phi \And \psi \Proves \Delta}{\Gamma, \phi \Proves \Delta}
\qquad
\infer{\Gamma,\psi \And \phi \Proves \Delta}{\Gamma, \phi \Proves \Delta}
\qquad
\infer{ \Gamma \Proves \Delta, \phi \And \psi}{\Gamma \Proves \Delta,\phi & \Gamma \Proves \Delta, \psi}
$$

$$
\infer{\Gamma, \phi \Or \psi \Proves \Delta}{\Gamma, \phi \Proves \Delta &\Gamma, \psi  \Proves \Delta}
\qquad
\infer{\Gamma \Proves \Delta,\phi \Or \psi}{\Gamma \Proves \Delta,\phi}
\qquad
\infer{\Gamma \Proves \Delta,  \psi \Or\phi}{\Gamma \Proves \Delta,\phi}
$$

$$
\infer{\Gamma, \exists v\: \phi \Proves \Delta}{\Gamma, \phi(v/w) \Proves \Delta}
\qquad
\infer{\Gamma \Proves \Delta, \exists v \: \phi}{\Gamma \Proves \Delta,  \phi(v/t)}
$$

$$
\infer{\Gamma, \forall v \in s\: \phi  \Proves \Delta}{\Gamma,  t \not \in s \Or \phi(v/t) \Proves \Delta}
\qquad
\infer{\Gamma \Proves \Delta,  \forall v \in s \: \phi}{\Gamma \Proves \Delta, w \not \in s \Or \phi(v/w)}
$$

\vspace{7mm}

\begin{tabular}{ccc}
Axiom Rule  & \hspace{30mm} & Weakening Rules \\
&  & \\
\infer{\phi \Proves \phi}{} & & \infer{\Gamma, \phi \Proves \Delta}{ \Gamma \Proves \Delta}
\qquad
\infer{\Gamma \Proves \Delta, \phi}{\Gamma \Proves \Delta}
\end{tabular}

\vspace{7mm}

\hspace{15mm}\begin{tabular}{ccc}
Multiset Rule  & \hspace{10mm} & $T$-Cut Rule \\
 &  &  \\
$\infer{\Gamma' \Proves \Delta'}{\Gamma \Proves \Delta}$ &   & $\infer{\Gamma \Proves \Delta, \Delta'}{\Gamma \Proves \Delta, \phi(t_1, \ldots, t_n) & \Gamma, \psi(t_1, \ldots, t_n) \Proves \Delta'}$ 
\end{tabular}

\caption{the system $\lL\kK_\Sigma(T)$}
\label{figure 2}

\vspace{10mm}

{\justifying \noindent For the left existential quantifier rule, the variable $w$ must not be free in $\Gamma$, $\Delta$, or $\phi$. For the right universal quantifier rule, the variable $w$ must not be free in $\Gamma$, $\Delta$, $\phi$, or $s$. For the multiset rule, the sequent $\Gamma'$ must consists of the same formulas as the sequent $\Gamma$, and the sequent $\Delta'$ must consist of the same formulas as $\Delta$. For the $T$-cut rule, the implication $\phi \Yields \psi$ must be an axiom of $T$.
\par}

\end{figure}

\newpage

\begin{figure}

\vspace{20mm}

\begin{enumerate}
\item[(0)] $\phi(t_1, \ldots, t_n) \YIELDS \psi(t_1, \ldots, t_n)$ \hfill for $\phi \Yields \psi$ an axiom of $T$
\item $\phi \YIELDS \Truth$
\item $\Falsehood \YIELDS \phi$
\item $\phi \And \psi \YIELDS \phi$
\item $\psi \And \phi \YIELDS \phi$
\item $\phi \YIELDS \phi \And \phi$
%\item $(\chi \Or \phi) \And (\chi \Or \psi) \YIELDS \chi \Or (\phi \And \psi)$
\item $\phi \YIELDS \phi \Or \psi$
\item $\phi \YIELDS \psi \Or \phi$
\item $\phi \Or \phi \YIELDS \phi$
\item $\chi \And (\phi \Or \psi) \YIELDS (\chi \And \phi) \Or (\chi \And \psi)$
\item $\phi(v/t) \YIELDS \exists v \: \phi $
\item $\exists v\: \phi \YIELDS \phi$ \hfill for $v$ not free in $\phi$
\item $\chi \And \exists v\: \phi \YIELDS \exists v\: (\chi \And \phi)$ \hfill for $v$ not free in $\chi$
\item $ \forall v \in s \: \phi \YIELDS t \not \in s \Or \phi(v/t)$
\item $ \phi \YIELDS \forall v \in s\: \phi$ \hfill for $v$ not free in $\phi$
\item[(14a)] $\forall v \in s \: (v \not \in s \Or \phi) \YIELDS \forall v \in s\: \phi$
\item $\forall v \in s \: (\chi \Or \phi) \YIELDS \chi \Or (\forall v \in s \: \phi)$ \hfill for $v$ not free in $\chi$
\end{enumerate}

\caption{rules the rewritting system $\mathrm{RK}_\Sigma(T)$}
\label{figure 3}

\vspace{10mm}

{\justifying \noindent The symbols $s, t, t_1, \ldots, t_n$ denote terms. Implications of the form (1)-(15) are the logical axioms of $\mathrm{RK}_\Sigma(T)$. The class of logical axioms is closed under legal substitution. Implications of the form (0) are substitution instances of the nonlogical axioms of $\mathrm{RK}_\Sigma(T)$.
For each implication $\sigma \YIELDS \sigma'$ of the form (0)-(15), the system $\mathrm{RK}_\Sigma(T)$ has the inference rule 
$$\infer{\Xi(P/\sigma')}{\Xi(P/\sigma)},$$
where $\Xi$ is any $\Sigma(\S)$ formula with a single schematic variable $P$ in place of a subformula.\par}
\end{figure}

\vspace{15mm}

\newpage \quad
\newpage

\newpage
\newpage

\newpage
\quad
\newpage
\quad \newpage\

\begin{figure}
$$\infer{\chi \And(\phi \Or \psi) \Proves (\chi \And \phi) \Or (\chi \And \psi)}{\infer{\chi, \phi \Or \psi \Proves (\chi \And \phi) \Or (\chi \And \psi)}
{\infer{\chi, \phi \Proves (\chi \And \phi) \Or (\chi \And \psi)}{\infer{\chi, \phi \Proves \chi \And \phi}{\infer{\chi, \phi \Proves \chi}{\chi \Proves \chi} & \infer{\chi, \phi \Proves \phi}{\phi \Proves \phi}}} & \infer{\chi, \psi \Proves (\chi \And \phi) \Or (\chi \And \psi)}{\infer{\chi, \psi \Proves \chi \And \psi}{\infer{\chi, \psi \Proves \chi}{\chi \Proves \chi} & \infer{\chi, \psi \Proves \psi}{\psi \Proves \psi}}}}}
$$
\caption{rule (9) derived in $\lL\iI$}\label{figure 4}
\end{figure}

\begin{figure}
$$
\infer{\chi \And \exists v\: \phi \Proves \exists v\: \chi \And \phi}{
\infer{\chi, \exists v \: \phi \Proves \exists v\: \chi \And \phi}{
\infer{\chi, \phi \Proves \exists v\: \chi \And \phi}{
\infer{\chi, \phi \Proves \chi \And \phi}{
\infer{\chi, \phi \Proves \chi}{\chi \Proves \chi} & \infer{\chi, \phi \Proves \phi}{\phi \Proves \phi}
}}}}
$$

\caption{rule (12) derived in $\lL\iI$}

\end{figure}

\begin{figure}
$$\infer{\overline T, \phi \Proves \forall v \in s\: \phi}{\infer{\phi \Proves \forall v \in s\: \phi}{\infer{\phi \Proves v \not \in s \Or \phi}{\phi \Proves \phi}  }  }$$
\caption{rule (14) derived in $\lL\iI$}
\end{figure}

\begin{figure}
$$\infer{ \forall v \in s\: (v \not \in s \Or \phi) \proves \forall v \in s\: \phi}{\infer{\forall v \in s\: (v \not \in s \Or \phi) \Proves v \not \in s \Or \phi}{\infer{v \not \in s \Or(v \not \in s \Or \phi) \Proves v \not \in s \Or \phi}{\infer {v \not \in s \Proves v \not \in s \Or \phi}{v \not \in s \Proves v \not \in s} & v \not \in s \Or \phi \Proves v \not \in s \Or \phi}    }}$$
\caption{rule (14a) derived in $\lL\iI$}
\end{figure}

\begin{figure}
\begin{enumerate}
\item $\Truth \YIELDS x = x$
\item $x = y \YIELDS y = x$
\item $x = y \And y = z  \YIELDS x = z$
\item $\phi(x, z_1, \ldots z_k) \And x= y \Yields \phi(y, z_1, \ldots, x_k)$ \hfill for $\phi$ any $\Sigma(\S)$ formula
\item $\Truth \YIELDS x =y \Or x \neq y$
\item $x = y \And x \neq y \YIELDS \Falsehood$
\end{enumerate}

\caption{the axioms for equality and inequality for vocabulary $\S$}\label{figure 8}

\vspace{0.5in}

\end{figure}

\begin{figure}
\begin{enumerate}\setcounter{enumi}{6}
\item $\Truth \YIELDS x < y \Or x \not < y$
\item $x< y \And  x \not < y \YIELDS \Falsehood$
\item $ x < \underline 0 \YIELDS \Falsehood$
\item $ \Truth \YIELDS x< \underline S(x)$
\item $ x < \underline S(y) \YIELDS x< y \Or x = y$
\item $\underline S(x) = \underline S(y) \YIELDS x =y$
\item $x \neq 0 \YIELDS \exists y\: x = \underline S(y)$
\item $\Truth \YIELDS \underline Z(x) = 0$
\item $\Truth \YIELDS \underline P^k_i(x_1, \ldots, x_k) = x_i$
\item $\Truth \YIELDS C(\underline g, \underline h_1, \ldots, \underline h_m)(x_1, \ldots,x_k) = \underline g( \underline h_1(x_1, \ldots, x_k), \ldots , \underline h_m(x_1, \ldots, x_k))$
\item $\Truth \YIELDS R(\underline g, \underline h) (0, x_1, \ldots, x_k) = g(x_1, \ldots, x_k)$
\item $\Truth \YIELDS R(\underline g , \underline h)(\underline S(y), x_1, \ldots, x_k) = h(y, R(\underline g, \underline h)(y, x_1, \ldots, x_k), x_1, \ldots, x_k)$
\item $\forall x < y\: \theta(x, z_1, \ldots, z_k) \Implies \theta(\underline S(x), z_1, \ldots, z_k)) \YIELDS \theta(0,z_1, \ldots, z_k)) \Implies \theta(y, z_1, \ldots, z_k))$
\end{enumerate}

\caption{the axioms of $\mathrm{PRA}$}\label{figure 9}

\vspace{10mm}

{\justifying \noindent

The $\Sigma(\S)$ theory $\mathsf{PRA}$ consists of the axioms in figure \ref{figure 8} and the axioms in figure \ref{figure 9}. The vocabulary $\S$ is the vocabulary of $\mathsf{PRA}$ given in Simpson's \textit{Subsystems of Second Order Arithmetic} \cite{Simpson}*{section IX.3}, together with binary relation symbols $\neq$, $<$, and $\not <$. In axiom scheme (22),
the formula $\theta(x, z_1, \ldots, z_k)$ is any $\Delta_0$ formula in the vocabulary $\S$. A formula is $\Delta_0$ if both its existential quantifiers and its universal quantifiers are bounded. If $\phi$ and $\psi$ are $\Delta_0$ formulas, then the material implication $\phi \Implies \psi$ is an abbreviation for the formula $\tilde \phi \Or \psi$, where $\tilde \phi$ is the $\Delta_0$ formula obtained from $\phi$ by replacing each quantifier, connective, and relation symbol with its dual: $\Falsehood \,\leftrightsquigarrow\, \Truth$, $\Or \,\leftrightsquigarrow\, \And$, $\exists \,\leftrightsquigarrow\, \forall$, $= \,\leftrightsquigarrow\, \neq$, and $< \,\leftrightsquigarrow\, \not <$.

\par

}

\end{figure}

\begin{figure}
\begin{enumerate}\setcounter{enumi}{6}
\item $\Truth \YIELDS x \in y \Or x \not \in y$
\item $x \in y \And x \not \in y \YIELDS \Falsehood$
\item $(\forall z \in x\: z \in y) \And (\forall z \in y\: z \in x) \YIELDS x =y$
\item $z \in y \And y \in x \YIELDS z \in \underline U(x)$
\item $z \in \underline U(x) \YIELDS \exists y \in x\: z \in y$
\item $\Truth \YIELDS x \in \underline A (x,y)$
\item $\Truth \YIELDS y \in \underline A(x,y)$
\item $z \in \underline A (x,y) \YIELDS z = x \Or z = y$
\item $\Truth \YIELDS \underline P^k_i(x_1, \ldots, x_k) = x_i$
\item $\Truth \YIELDS C(\underline g, \underline h_1, \ldots, \underline h_m)(x_1, \ldots,x_k) = \underline g( \underline h_1(x_1, \ldots, x_k), \ldots , \underline h_m(x_1, \ldots, x_k))$
\item $w \in y \YIELDS \underline f(w, x_1, \ldots, x_k) \in I(\underline f)(y, x_1, \ldots, x_k)$
\item $ z \in I(\underline f)(y, x_1, \ldots, x_k) \YIELDS \exists w \in y\: \underline f(w, x_1, \ldots, x_k)$
\item $\Truth \YIELDS R(\underline f)(y, x_1, \ldots, x_k) = \underline f(I(R(\underline f))(y, x_1, \ldots, x_k), y, , x_1, \ldots, x_k)$
\item $z \in y \And \phi(z, x_1, \ldots, x_k) \YIELDS z \in S(\phi)(y, x_1, \ldots, x_k) $
\item $z \in S(\phi)(y, x_1, \ldots, x_k) \YIELDS z \in y$
\item $z \in S(\phi)(y, x_1, \ldots, x_k) \YIELDS \phi(z, x_1, \ldots, x_k)$
\item $\exists z \in x\: \Truth \YIELDS \exists z \in x\: \forall y \in x\: z \not \in y$

\end{enumerate}

\caption{the axioms of $\mathrm{PRS}$}\label{figure 10}

\vspace{6mm}

{\justifying \noindent

The $\Sigma(\S)$ theory $\mathsf{PRS}$ consists of the axioms in figure \ref{figure 8} and the axioms in figure \ref{figure 10}. It is a variant of the theory $\mathsf{PRS}$ introduces in Rathjen's \textit{A proof-theoretic characterization of the primitive recursive set functions} \cite{Rathjen}*{section 6}. The vocabulary $\S$ has binary relation symbols $=$, $\neq$, $\in$, and $\not \in$; it has a unary function symbol $\underline U$, a binary function symbol $\underline A$, and function symbols $\underline P_k^i$ for positive integers $k$ and $1 \leq i \leq k$; it has a $k$-ary function symbol $C(\underline g, \underline h_1, \ldots, \underline h_m)$ for each $m$-ary function symbol $g$ and $k$-ary function symbols $\underline h_1, \ldots, \underline h_n$; it has a $(k+1)$-ary function symbol $I(\underline f)$ for each $(k+1)$-ary function symbol $\underline f$; it has a $(k+1)$-ary function symbol $R(\underline f)$ for each $(k+2)$-ary function symbol $\underline f$; finally, it has a $(k+1)$-ary function symbol $S(\phi)$ for each $\Delta_0$ formula $\phi(x_0, \ldots, x_k)$.  In this context, a formula $\phi$ is implicitly equipped with a positive integer arity, specifying its free variables $x_0, x_1, \ldots, x_k$, some of which may fail to appear in the formula. A formula is $\Delta_0$ if both its existential quantifiers and its universal quantifiers are bounded.
\par
}

\vspace{14mm}

\end{figure}

\vspace{7mm}

\begin{figure}

\begin{enumerate}\setcounter{enumi}{23}
\item $\Truth \YIELDS \underline \exists x \in \underline N(w)\: \Truth$
\item $y \in x \And x \in \underline N(w) \YIELDS y \in \underline N(w)$
\item $z \in y \And y \in x \And x \in \underline N(w) \YIELDS z \in x$
\item $x \in \underline N(w) \YIELDS \exists y\in x\: \forall z \in x\: z \in y \Or z = y$
\end{enumerate}

\caption{the axioms of infinity}\label{figure 11}
\vspace{6mm}

{\justifying \noindent
The vocabulary $\S$ is extended by the unary function symbol $\underline N$. The vocabulary $\S$ is then closed under the constructions detailed in the caption to figure \ref{figure 10}. \par}

\end{figure}

\end{document}